\documentclass[10pt,reqno]{amsart}

\usepackage[a4paper,left=35mm,right=35mm,top=30mm,bottom=30mm,marginpar=25mm]{geometry}

\usepackage{amsmath}
\usepackage{amssymb}
\usepackage{amsthm}
\usepackage[utf8]{inputenc}
\usepackage{eurosym}
\usepackage{graphicx}
\usepackage{hyperref}
\usepackage{dsfont}
\usepackage{dsfont}
\usepackage{xcolor,colortbl}
\usepackage{comment}

\allowdisplaybreaks

\usepackage{hyperref}

\usepackage{ifthen}

\makeindex

\newcommand{\tr}{\operatorname{tr}}

\renewcommand{\div}{\operatorname{div}}

\newcommand{\Rr}{{\mathbb{R}}}

\newcommand{\Ee}{{\mathds{E}}}

\newcommand{\Aa}{{\mathcal{A}}}

\newcommand{\Pp}{{\mathcal{P}}}

\newcommand{\bX}{{\bf X}}

\newcommand{\bb}{{\bf b}}
\newcommand{\bsigma}{{\boldsymbol \sigma}}

\newcommand{\bv}{{\bf v}}

\def\leq{\leqslant}

\numberwithin{equation}{section}

\newtheoremstyle{thmlemcorr}{10pt}{10pt}{\itshape}{}{\bfseries}{.}{10pt}{{\thmname{#1}\thmnumber{
#2}\thmnote{ (#3)}}}
\newtheoremstyle{thmlemcorr*}{10pt}{10pt}{\itshape}{}{\bfseries}{.}\newline{{\thmname{#1}\thmnumber{
#2}\thmnote{ (#3)}}}
\newtheoremstyle{defi}{10pt}{10pt}{\itshape}{}{\bfseries}{.}{10pt}{{\thmname{#1}\thmnumber{
#2}\thmnote{ (#3)}}}
\newtheoremstyle{remexample}{10pt}{10pt}{}{}{\bfseries}{.}{10pt}{{\thmname{#1}\thmnumber{
#2}\thmnote{ (#3)}}}
\newtheoremstyle{ass}{10pt}{10pt}{}{}{\bfseries}{.}{10pt}{{\thmname{#1}\thmnumber{
A#2}\thmnote{ (#3)}}}

\theoremstyle{thmlemcorr}
\newtheorem{theorem}{Theorem}
\numberwithin{theorem}{section}
\newtheorem{lemma}[theorem]{Lemma}

\theoremstyle{thmlemcorr*}
\newtheorem{theorem*}{Theorem}
\newtheorem{lemma*}[theorem]{Lemma}
\newtheorem{corollary*}[theorem]{Corollary}
\newtheorem{proposition*}[theorem]{Proposition}
\newtheorem{problem*}[theorem]{Problem}
\newtheorem{conjecture*}[theorem]{Conjecture}

\theoremstyle{defi}
\newtheorem{definition}[theorem]{Definition}

\newtheorem{problem}{Problem}

\theoremstyle{remexample}
\newtheorem{remark}[theorem]{Remark}

\theoremstyle{ass}

\begin{document}

\title{A Mean Field Game price model with noise}

\author{Diogo Gomes}\thanks{King Abdullah University of Science and Technology (KAUST), CEMSE Division, Thuwal 23955-6900. Saudi Arabia. e-mail: diogo.gomes@kaust.edu.sa.}%
\author{Julian Gutierrez}\thanks{King Abdullah University of Science and Technology (KAUST), CEMSE Division, Thuwal 23955-6900. Saudi Arabia. e-mail: julian.gutierrezpineda@kaust.edu.sa.}%
\author{Ricardo Ribeiro}\thanks{King Abdullah University of Science and Technology (KAUST), CEMSE Division, Thuwal 23955-6900. Saudi Arabia. e-mail: ricardo.ribeiro@kaust.edu.sa.}
	\thanks{Universidade Tecnol\'ogica Federal do Paran\'a (UTFPR), Departamento Acad\^emico de Matem\'atica, Avenida dos Pioneiros 3131, 86036-370 Londrina, PR, Brazil. e-mail: ricardoribeiro@utfpr.edu.br.}%



\keywords{Mean Field Games; Price formation; Common noise}

\thanks{
      The authors were partially supported by KAUST baseline funds and 
 KAUST OSR-CRG2017-3452.
}
\date{\today}

\begin{abstract}

In this paper, we propose a mean-field game model for the price formation of a commodity whose production is subjected to random fluctuations. 
The model generalizes existing deterministic price formation models. 

Agents seek to minimize their average cost by choosing their trading rates with a price that is characterized by a balance between supply and demand.
The supply and the price processes are assumed to follow stochastic differential equations. 

Here, we show that, for linear dynamics and quadratic costs, the optimal trading rates are determined in feedback form. 
Hence, the price arises as the solution to a stochastic differential equation, whose coefficients depend on the solution of a system of ordinary differential equations.

\end{abstract}

\maketitle

\section{Introduction}

Mean-field games (MFG) is a tool to study  the Nash equilibrium of infinite populations of rational agents. These agents select their actions based on their state and the statistical information about the population. Here, we study a price formation model for a commodity traded in a market under uncertain supply, which is a common noise shared by the agents. These agents are rational and aim to minimize the average trading cost by selecting their trading rate. The distribution of the agents solves a stochastic partial differential equation. Finally, a market-clearing condition characterizes the price.

We consider a commodity whose supply process is described by a  stochastic differential equation; that is, we are given a drift 
 $b^S:[0,T]\times \Rr^2\to \Rr$ and volatility $\sigma^S:[0,T]\times \Rr^2\to \Rr^+_0$, 
 which are smooth functions, and the supply $Q_s$ is 
determined by the 
stochastic differential equation 
\begin{equation}\label{eq: supply}
dQ_s = b^S(Q_s,\varpi_s,s)ds+ \sigma^S(Q_s,\varpi_s,s)dW_s \quad \mbox{ in } [0,T]      
\end{equation}
with the initial condition $\bar{q}$. We would like to determine  the drift $b^P:[0,T]\times \Rr^2\to \Rr$, the volatility $\sigma^P:[0,T]\times \Rr^2\to \Rr^+_0$, and $\bar{w}$ such that 
 the price $\varpi_s$ solves 
\begin{equation}\label{eq: price}
d\varpi_s=b^P(Q_s,\varpi_s,s)ds+\sigma^P(Q_s,\varpi_s,s)dW_s \quad \mbox{ in } [0,T]
\end{equation}
with initial condition $\bar{w}$ and such that it ensures a market clearing condition. It may not be possible to find $b^P$ and $\sigma^P$ in a feedback form. However, for linear dynamics, as we show here, we can solve quadratic models, which are of great interest in applications.

Let $\bX_s$ be the quantity of the commodity held by an agent at time $s$ for $t \leq s \leq T$. 
This agent  trades this commodity, controlling its rate of change, $\bv$, thus
\begin{equation}\label{eq: commodity quantity}
d\bX_s = \bv(s) ds ~~\mbox{ in } [t,T].
\end{equation}
At time $t$, an agent who holds $x$ and observes $q$ and $w$ chooses a 
progressively measurable control process $\bv$ to minimize 
the expected cost functional 
\begin{equation}\label{eq:functional}
J(x,q,w,t;\bv) = \Ee \left[ \int_t^T L(\bX_s,\bv(s))+\varpi_s \bv(s) ds + \Psi \left(\bX_ T,Q_T,\varpi_T\right)\right],
\end{equation}
subject to the dynamics \eqref{eq: supply}, \eqref{eq: price}, and \eqref{eq: commodity quantity}
with initial condition $\bX_t=x$, 
 and the expectation is taken w.r.t. the standard filtration generated by the Brownian motion. The Lagrangian, $L$, takes into account costs such as market impact or storage, and the terminal cost $\Psi$ stands for the terminal preferences of the agent.

%
%
%
%
This control problem determines a Hamilton-Jacobi equation addressed in Section \ref{hjevt}. In turn, each agent selects an optimal control and uses it to adjust its holdings. Because the source of noise in $Q_t$ is common to all agents, the evolution of the  probability distribution of agents is not deterministic. Instead, it is given by a stochastic transport equation derived in Section \ref{stochastic transport}. Finally, the price is determined by a market-clearing condition that ensures that supply meets demand. We study this condition in Section \ref{balance condition section}. 

Mathematically, the price model corresponds to the following problem.

\begin{problem}\label{problem: problem}

        Given a Hamiltonian, $H: \Rr^2\to \Rr$, $H\in C^\infty$, 
        a commotity's supply initial value, $\bar q\in \Rr$, supply drift, $b^S:\Rr^2\times [0,T]\to \Rr$,  and supply volatility, $\sigma^S:\Rr^2\times [0,T]\to \Rr$, 
          a terminal cost, $\Psi:\Rr^3\to \Rr$, $\Psi\in C^\infty(\Rr^3)$, and an initial 
           distribution of agents, $\bar m\in C^\infty_c(\Rr)\cap \Pp(\Rr)$,  
        find $u:\Rr^3\times [0,T]\to \Rr$, $\mu\in C([0,T],\Pp(\Rr^3))$, $\bar w\in\Rr$, the price at $t=0$, the price drift
         $b^P:\Rr^2\times[0,T] \to \Rr$, and the price volatility $\sigma^P:\Rr^2\times [0,T]\to \Rr$ solving
        \begin{equation}\label{eq: problem}
                \begin{cases}
                -u_t+H(x,w+u_x)=b^S u_q +b^P u_w +\tfrac{1}{2}(\sigma^S)^2u_{qq}+\sigma^S\sigma^P u_{qw}+\tfrac{1}{2}(\sigma^P)^2u_{ww}\\
                d \mu_t = \left(\left(\mu\tfrac{(\sigma^S)^2 }{2}\right)_{qq}+(\mu\sigma^S\sigma^P )_{qw}+\left(\mu\tfrac{(\sigma^P)^2 }{2}\right)_{ww}-\div (\mu\bb )\right) dt-\div(\mu\bsigma) dW_t\\
                \int_{\Rr^3}q+D_pH(x,w+u_x(x,q,w,t))\mu_t(dx\times dq\times dw)=0,
                \end{cases}
        \end{equation}
        and the terminal-initial conditions
\begin{equation}\label{eq: terminal-initial conditions}
        \begin{cases}
                u(x,q,w,T)=\Psi(x,q,w)\\
                \mu_0= \bar m \times \delta_{\bar q} \times \delta_{\bar w},           
        \end{cases}
\end{equation}
        where $\bb=(-D_pH(x,w+u_x),b^S,b^P)$, $\bsigma=(0,\sigma^S,\sigma^P)$, and the divergence is taken w.r.t. $(x,q,w)$.
\end{problem}

Given a solution to the preceding problem, we construct the supply and price processes 
\[
Q_t=\int_{\Rr^3} q ~ \mu_t(dx\times dq\times dw)
\]
and 
\[
\varpi_t=\int_{\Rr^3} w ~ \mu_t(dx\times dq\times dw),
\]
which also solve
\[
\begin{cases}
dQ_t&=b^S(Q_t,\varpi_t,t)dt+ \sigma^S(Q_t,\varpi_t,t)dW_t \quad \mbox{ in } [0,T]      \\
d\varpi_t&=b^P(Q_t,\varpi_t,t)dt+\sigma^P(Q_t,\varpi_t,t)dW_t \quad \mbox{ in } [0,T]
\end{cases}
\]
with initial conditions
\begin{equation}\label{eq: Balance first appearence}
        \begin{cases}
                Q_0 = \bar{q}\\
                \varpi_0 = \bar{w}
        \end{cases}
\end{equation}
and satisfy the market-clearing condition
\[Q_t=\int_{\Rr} -D_pH(x, \varpi_t + u_x(x,Q_t,\varpi_t,t)) \mu_t(dx).\]

In \cite{gomes2018mean}, the authors presented a model where the supply for the commodity was a given deterministic function, and the balance condition between supply and demand gave rise to the price as a Lagrange multiplier. Price formation models were also studied by Markowich et al. \cite{MR2573148}, Caffarelli et al. \cite{MR2817378}, and Burger et al. \cite{MR3078202}. 
The behavior of rational agents
that control an electric load  
was considered 
in \cite{MR1029068, MR933047}.  
For example, 
turning on or off space heaters controls the electric load
as was discussed in 
\cite{Kizilkale2014829, Kizilkale20141867,Kizilkale20143493}. 
Previous authors addressed price formation when 
the demand is a given function of the price \cite{GLL10}
or that the price is a function of the demand, see, for example
\cite{2011arXiv1110.1732C}, \cite{2017arXiv171008991C}, 
\cite{de2016distributed}, \cite{PaolaTrovatoAngeli2019}, 
and \cite{2017arXiv170707853J}. 
An $N$-player version of an economic growth model was 
presented in \cite{Gomes20164693}.

Noise in the supply together with a balance condition is a central issue
in price formation 
that could not be handled directly with the techniques in previous papers. A probabilistic approach of the common noise is discussed in Carmona et al. in  \cite{carmona2016}. Another approach is through the master equation, involving derivatives with respect to measures, which can be found in \cite{carmona2014master}. None of these references, however, addresses problems with integral constraints such as \eqref{eq: Balance first appearence}.

Our model corresponds to the one in \cite{gomes2018mean}  for the deterministic setting when we take the volatility for the supply to be 0. Here, we study the linear-quadratic case, that is, when the cost functional is quadratic, and the dynamics \eqref{eq: supply} and \eqref{eq: price} are linear. In Section \ref{subsection: quadratic solutions}, we provide a constructive approach to get semi-explicit solutions of price models for linear dynamics and quadratic cost. This approach avoids the use of the master equation. The paper ends with a brief presentation of simulation results in Section \ref{section: simulation}.


\section{The Model}
In this section, we derive Problem \ref{problem: problem} from the price model. We begin with standard tools of optimal control theory. Then, we derive the stochastic transport equation, and we end by introducing the market-clearing (balance) condition.
\subsection{Hamilton-Jacobi equation and verification theorem}
\label{hjevt}

The \emph{value function} for an agent who at time $t$ holds an amount $x$ of the commodity, whose instantaneous supply and price are $q$ and $w$,  is
\begin{equation}\label{eq:value}
u(x,q,w,t)= \inf_{\bv} J(x,q,w,t;\bv)
\end{equation}
where $J$ is given by \eqref{eq:functional} and the infimum is taken over the set $\Aa\left((t,T]\right)$ of  all progressively  measurable functions $v:[t,T]\to \Rr$. 
Consider the Hamiltonian, $H$, which is the Legendre transform of $L$; that is, for $p\in \Rr$,
\begin{equation}\label{eq: Hamiltonian}
H(x,p)=\sup_{v\in \Rr} [-pv-L (x,v)].
\end{equation}

Then, from standard stochastic optimal control theory, whenever $L$ is strictly convex,  if $u$ is $C^2$, it solves the Hamilton-Jacobi equation in $\Rr^3 \times [0,T)$
\begin{equation}\label{eq:hamilton-jacobi}
-u_t+H(x,w+u_x) -b^Pu_w -b^S u_q - \tfrac{(\sigma^P)^2}{2}u_{ww}- \tfrac{(\sigma^S)^2}{2}u_{qq}- \sigma^P\sigma^Su_{wq}=0
\end{equation}
with the terminal condition 
\begin{equation}\label{eq: final condition}
        u(x,q,w,T)= 
        \Psi \left(x,q,w\right).
\end{equation}
Moreover, as the next verification theorem establishes, any $C^2$ solution 
of \eqref{eq:hamilton-jacobi} is the value function.
\begin{theorem}[Verification]\label{Thm: Verification}
Let $\tilde u:[0,T]\times \Rr^3 \to \Rr$ be a smooth solution of 
\eqref{eq:hamilton-jacobi} with terminal condition \eqref{eq: final condition}. 
Let $(\bX^*,Q,\varpi)$ solve \eqref{eq: commodity quantity}, \eqref{eq: supply} and \eqref{eq: price}, where $\bX^*$ is driven by the progressively measurable control 
\begin{equation*}
       \bv^*(s):=-D_pH(\bX^*_s,\varpi_s+ \tilde{u}_x (\bX^*_s,Q_s,\varpi_s,s)).
\end{equation*}

Then
\begin{enumerate}
        \item $\bv^*$ is an optimal control for \eqref{eq:value} 
        \item $\tilde u=u$, the value function.
\end{enumerate}

\end{theorem}

\subsection{Stochastic transport equation}\label{stochastic transport}
Theorem \ref{Thm: Verification} provides an optimal feedback  strategy. As usual in MFG, we assume that the agents are rational and, hence, choose to follow this optimal strategy. This behavior gives rise to a flow that transports the agents and induces a random measure that encodes their distribution. Here, we derive a stochastic PDE solved by this random measure. To this end, let $u$ solve 
\eqref{eq:hamilton-jacobi} and consider the random flow associated with the diffusion
\begin{equation}\label{eq:diffusion}
\begin{cases}
        d\bX_s=-D_pH(\bX_s,\varpi_s+u_x(\bX_s,Q_s,\varpi_s,s)) ds\\
        dQ_s=b^S(Q_s,\varpi_s,s)ds+ \sigma^S(Q_s,\varpi_s,s)dW_s \\
        d\varpi_s=b^P(Q_s,\varpi_s,s)ds+\sigma^P(Q_s,\varpi_s,s)dW_s
\end{cases}
\end{equation}
with initial conditions
\begin{equation*}
\begin{cases}
        \bX_0=x\\
        Q_0=\bar{q}\\
        \varpi_0=\bar{w}.
\end{cases}
\end{equation*}
That is, for a given realization $\omega$ of the common noise, the flow maps the initial conditions $(x,\bar{q},\bar{w})$ to the solution of \eqref{eq:diffusion} at time $t$, which we denote by $\left(\bX_t^\omega (x,\bar{q},\bar{w}),Q_t^\omega (\bar{q},\bar{w}),\varpi_t^\omega (\bar{q},\bar{w})\right)$. Using this map, we define a measure-valued stochastic process $\mu_t$ as follows:

\begin{definition}\label{def:random measure}
Let $\omega$ denote a realization of the common noise $W$ on $0\leq s \leq T$. Given a measure $\bar{m}\in\Pp(\Rr)$ and initial conditions $\bar{q},\bar{w}\in\Rr$ take $\bar \mu \in \Pp(\Rr^3)$ by $\bar \mu=\bar{m}\times \delta_{\bar{q}}\times \delta_{\bar{w}}$ and define a random measure $\mu_t$ by the mapping $\omega \mapsto \mu^\omega_t \in \Pp(\Rr^3)$, where $\mu_t^\omega$ is characterized as follows:
for any bounded  and continuous function $\psi:\Rr^3\to \Rr$
\begin{align*}
&\int_{\Rr^3} \psi(x,q,w) \mu_t^\omega(dx\times dq\times dw) 
\\
&= \int_{\Rr^3} \psi \left(\bX_t^\omega (x,q,w), Q_t^\omega(q,w), \varpi_t^\omega(q,w) \right)\bar{\mu}(dx\times dq \times dw).
\end{align*}
\end{definition}

\begin{remark} 
Because $\bar \mu=\bar{m}\times \delta_{\bar{q}}\times \delta_{\bar{w}}$, we have
\begin{align*}
&\int_{\Rr^3} \psi \left(\bX_t^\omega (x,q,w), Q_t^\omega(q,w), \varpi_t^\omega(q,w) \right)\bar{\mu}(dx\times dq \times dw) 
\\
&= \int_{\Rr} \psi \left(\bX_t^\omega (x,\bar{q},\bar{w}), Q_t^\omega(\bar{q},\bar{w}), \varpi_t^\omega(\bar{q},\bar{w}) \right)\bar{m}(dx).
\end{align*}
Moreover, due to the structure of \eqref{eq:diffusion},
\[\mu_t^\omega=(\bX_t^\omega(x,\bar{q},\bar{w}) \# \bar{m})\times \delta_{Q_t^\omega(\bar{q},\bar{w})}\times \delta_{\varpi^\omega_t(\bar{q},\bar{w})}.\]
\end{remark}




\begin{definition} Let $\bar{\mu}\in \Pp(\Rr^3)$ and write 
\begin{align*}
\bb(x,q,w,s)&=(-D_pH(x,w+u_x(x,q,w,s)),b^S(q,w,s),b^P(q,w,s)), 
\\
\bsigma(q,w,s)&=(0,\sigma^S(q,w,s),\sigma^P(q,w,s)).
\end{align*}
A measure-valued stochastic process $\mu=\mu(\cdot,t)=\mu_t(\cdot)$ is a \emph{weak solution} of the stochastic PDE
        \begin{equation}\label{eq:SDEmeasure}
        d\mu_t =\left(-\div (\mu\bb )+\left(\mu\tfrac{(\sigma^S)^2 }{2}\right)_{qq}+(\mu\sigma^S\sigma^P )_{qw}+\left(\mu\tfrac{(\sigma^P)^2 }{2}\right)_{ww}\right) dt-\div(\mu\bsigma) dW_t,
        \end{equation}
with initial condition $\bar{\mu}$ if for any bounded smooth test function $\psi:\Rr^3\times[0,T]\to \Rr$ 
\begin{align}\label{eq: measure weak solution definition}
                &\int_{\Rr^3} \psi(x,q,w,t)\mu_t(dx\times dq \times dw)=\int_{\Rr^3} \psi(x,q,w,0)\bar{\mu}(dx\times dq \times dw)
                \\
                &+\int_0^t\int_{\Rr^3}\partial_t\psi+ D\psi \cdot \bb +\tfrac{1}{2}\tr \left(\bsigma^T \bsigma D^2\psi \right)\mu_s(dx \times dq \times dw) ds
                \\ 
                & + \int_0^t\int_{\Rr^3} D\psi \cdot \bsigma \mu_s(dx \times dq \times dw) dW_s,
        \end{align}
where the arguments for $\bb$, $\bsigma$ and $\psi$ are $(x,q,w,s)$ and the differential operators $D$ and $D^2$ are taken w.r.t. the spatial variables $x,q, w$.
\end{definition}

\begin{theorem}\label{theo:transport stochastic equation}
        Let $\bar{m}\in\Pp(\Rr)$ and $\bar{q},\bar{w}\in\Rr$. The random measure from Definition \ref{def:random measure} is a weak solution of the stochastic partial differential equation \eqref{eq:SDEmeasure}  with initial condition $\bar{\mu}=\bar{m}\times \delta_{\bar{q}}\times \delta_{\bar{w}}$.
\end{theorem}
\begin{proof}
        Let $\psi:\Rr^3\times [0,T]\to \Rr$ be a bounded smooth test function.  Consider the stochastic process $s \mapsto \int_{\Rr^3}\psi(x,q,w,s)\mu^\omega_s(dx\times dq \times dw)$. Let 
        \[(\bX_t(x,\bar{q},\bar{w}),Q_t(\bar{q},\bar{w}),\varpi_t(\bar{q},\bar{w}))\]
        be the flow induced by \eqref{eq:diffusion}. By the definition of $\mu^\omega_t$,
      \begin{align*}
      &\int_{\Rr^3}\psi(x,q,w,t)\mu_t^\omega(dx\times dq\times dw)-\int_{\Rr^3}\psi(x,q,w,0)\bar{\mu}(dx\times dq\times dw)
        \\
     &=\int_{\Rr} \left[\psi(\bX^\omega_t(x,\bar{q},\bar{w}),Q^\omega_t(\bar{q},\bar{w}),\varpi^\omega_t(\bar{q},\bar{w}),t)-\psi(x,\bar{q},\bar{w},0)\right]\bar{m} (dx).
      \end{align*}
Then, applying Ito's formula to the stochastic process 
\[s\mapsto \int_{\Rr} \psi(\bX_s(x,\bar{q},\bar{w}),Q_s(\bar{q},\bar{w}),\varpi_s(\bar{q},\bar{w}),s)\bar{m}(dx),\] the preceding expression becomes
      \begin{align*}
&\int_0^t d \Big( \int_{\Rr} \psi(\bX_s(x,\bar{q},\bar{w}),Q_s(\bar{q},\bar{w}),\varpi_s(\bar{q},\bar{w}),s)\bar{m}(dx) \Big) ds
        \\
		&=\int_0^t \int_{\Rr}\left[D_t\psi+D\psi \cdot\bb+   
        \tfrac{1}{2}\tr (\bsigma^T\bsigma D^2\psi )\right]\bar{m}(dx)ds 
        \\
        &+ \int_0^t \int_{\Rr} D\psi\cdot \bsigma\bar{m}(dx) dW_s 
        \\        
        &=\int_0^t\int_{\Rr^3}\left[D_t\psi+ D\psi \cdot\bb + \tfrac{1}{2}\tr (\bsigma^T \bsigma D^2\psi)
                \right]\mu_s(dx\times dq\times dw)ds
		\\
        &+\int_0^t\int_{\Rr^3} D\psi\cdot \bsigma \mu_s(dx\times dq\times dw) dW_s,
      \end{align*}
where arguments of $\bb$, $\bsigma$ and the partial derivatives of $\psi$ in the integral with respect to $\bar{m}(dx)$ are $(\bX_s(x,\bar{q},\bar{w}),Q_s(\bar{q},\bar{w}),\varpi_s(\bar{q},\bar{w}),s)$, and in the integral with respect to $\mu_t(dx\times dq \times dw)$ are $(x,q,w,t)$.  Therefore, 
      \begin{align*} 
            &\int_{\Rr^3}\psi(x,q,w,t)\mu_t^\omega(dx\times dq\times dw)-\int_{\Rr^3}\psi(x,q,w,0)\bar{\mu}(dx\times dq\times dw)
            \\
        &=\int_0^t\int_{\Rr^3}\left[D_t\psi+
                D\psi \cdot\bb +   
                \tfrac{1}{2}\tr (D^2\psi:(\bsigma,\bsigma))
                \right]\mu^\omega_s(dx\times dq\times dw)ds\\
        &+\int_0^t\int_{\Rr^3} 
        D\psi\cdot \bsigma \mu^\omega_s(dx\times dq\times dw) dW_s.
      \end{align*} 
Hence, \eqref{eq: measure weak solution definition} holds.
\end{proof}

%

\subsection{Balance condition}\label{balance condition section}
The balance condition requires the average trading rate to be equal to the supply. Because agents are rational and, thus, use their optimal strategy, this condition takes the form
\begin{align}\label{eq: balance condition}
Q_t&=\int_{\Rr^3} -D_pH(x, w + u_x(x,q,w,t)) \mu^\omega_t(dx\times dq \times dw),
\end{align}
where $\mu_t^\omega$ is given by Definition \ref{def:random measure}. Because $Q_t$ satisfies a stochastic differential equation, the previous can also be read in differential form as
\begin{align}\label{eq: balance condition}
b^S(Q_t,\varpi_t,t)dt+ \sigma^S(Q_t,\varpi_t,t)dW_t = d\int_{\Rr^3} -D_pH(x, w + u_x(x,q,w,t)) \mu^\omega_t(dx\times dq \times dw).
\end{align}
The former condition determines $b^P$ and $\sigma^P$. In general, $b^P$ and $\sigma^P$ are only progressively measurable and not in feedback form. In this case, 
the Hamilton--Jacobi  \eqref{eq:hamilton-jacobi} must be replaced by either a 
stochastic partial differential equation or the problem must be modeled by the master equation. However, as we discuss next, in the linear-quadratic case, we can 
find $b^P$ and $\sigma^P$ in feedback form.

\section{Potential-free Linear-quadratic price model}

Here, we consider a price model for linear dynamics and quadratic cost. The Hamilton-Jacobi equation admits quadratic solutions. Then, the balance equation determines the dynamics of the price, and the model is reduced to a first-order system of ODE.\\
Suppose that $L(x,v)= \tfrac{c}{2}v^2$ and, thus, $H(x,p)=\tfrac{1}{2c}p^2$. Accordingly, the corresponding MFG model is 
\begin{equation}\label{eq: quadratic hamiltonian mfg no potential}
\begin{cases} 
-u_t + \tfrac{1}{2c}(w+u_x)^2 -b^Pu_w -b^S u_q - \tfrac{1}{2}(\sigma^P)^2u_{ww}- \tfrac{1}{2}(\sigma^S)^2u_{qq}- \sigma^P\sigma^Su_{wq}=0
\\
d \mu_t = \left(\left(\mu\tfrac{(\sigma^S)^2 }{2}\right)_{qq}+(\mu\sigma^S\sigma^P )_{qw}+\left(\mu\tfrac{(\sigma^P)^2 }{2}\right)_{ww}-\div (\mu\bb )\right) dt-\div(\mu\bsigma) dW_t
\\
Q_t
=-\tfrac{1}{c} \varpi_t +\int_{\Rr} -\tfrac{1}{c} u_x(x,q,w,t) \mu_t^\omega(dx\times dq \times dw).
\end{cases}
\end{equation}
Assume further that $\Psi$ is quadratic; that is,
\[
\Psi(x,q,w)=c_0+c_1^1x+c_1^2q + c_1^3w + c_2^1x^2 + c_2^2xq + c_2^3xw + c_2^4q^2 + c_2^5qw + c_2^6w^2.
\] 
\subsection{Balance condition}
Let
\[\Pi_t=\int_{\Rr^3} u_x(x,q,w,t) \mu_t(dx\times dq\times dw).\]
The balance condition  is $Q_t=-\tfrac{1}{c} \left(\varpi_t+\Pi_t\right)$. Furthermore, Definition \ref{def:random measure}  provides the identity
\[\Pi_t=\int_{\Rr} u_x(\bX_t^*(x,\bar{q},\bar{w}),Q_t(\bar{q},\bar{w}),\varpi_t(\bar{q},\bar{w}),t) \bar{m}(dx).\]
\begin{lemma}\label{lem: balance}
	Let $(\bX^*,Q,\varpi)$ solve \eqref{eq: commodity quantity}, \eqref{eq: supply} and \eqref{eq: price} with $\bv=\bv^*$, the optimal control, and initial conditions $\bar{q},\bar{w}\in\Rr$. Let $u\in C^3(\Rr^3\times[0,T])$ solve the Hamilton-Jacobi equation \eqref{eq:hamilton-jacobi}. Then
\begin{align}\label{eq: Balance condition auxiliary term}
d\Pi_t=&
\int_{\Rr} \left(u_{xq} \sigma^S + u_{xw} \sigma^P  \right)\bar{m}(dx)dW_t,
\end{align}
where the arguments for the partial derivatives of $u$ are $(\bX_t^*(x,\bar{q},\bar{w}),Q_t(\bar{q},\bar{w}),\varpi_t(\bar{q},\bar{w}),t)$.
\end{lemma}

\begin{proof}
By Itô's formula, the process $t\mapsto u_x(\bX^*_t,Q_t,\varpi_t,t)$ solves
\begin{align}\label{eq: Ito diff quadratic}
&d\big(u_x(\bX^*_t,Q_t,\varpi_t,t)\Big)\nonumber
\\
&= \Big( u_{xt} + u_{xx}\bv^* + u_{xq} b^S + u_{xw} b^P + u_{xqq} \tfrac{1}{2} (\sigma^S)^2 + u_{xqw} \sigma^S \sigma^P + u_{xww} \tfrac{1}{2} (\sigma^P)^2 \Big) dt +\nonumber \\
&+ \Big( u_{xq}\sigma^S + u_{xw} \sigma^P \Big) dW_t,
\end{align}
with $\bv^*(t)=-\tfrac{1}{c}(\varpi_t+u_x(\bX^*_t,Q_t,\varpi_t,t))$. By differentiating the Hamilton-Jacobi equation, we get 
\begin{equation*}
-u_{tx}+\tfrac{1}{c}(\varpi_t+u_x)u_{xx} -b^Pu_{wx} -b^S u_{qx} - \tfrac{(\sigma^P)^2}{2}u_{wwx}- \tfrac{(\sigma^S)^2}{2}u_{qqx}- \sigma^P\sigma^Su_{wqx}=0.
\end{equation*}
Substituting the previous expression in \eqref{eq: Ito diff quadratic}, we have
\begin{align*}
&d\Big( \int_{\Rr} u_x(\bX^*_t(x,\bar{q},\bar{w}),Q_t(\bar{q},\bar{w}),\varpi_t(\bar{q},\bar{w}),t)\bar{m}(dx)\Big)
\\
&= \int_{\Rr} \left(\tfrac{1}{c}(\varpi_t+u_x)u_{xx} +u_{xx} \bv^* \right)\bar{m}(dx)dt +\int_{\Rr} \left(u_{xq} \sigma^S + u_{xw} \sigma^P  \right)\bar{m}(dx)dW_t.
\end{align*}
The preceding identity simplifies to
\begin{gather*}
\int_{\Rr} \left(u_{xq} \sigma^S + u_{xw} \sigma^P  \right)\bar{m}(dx)dW_t.\qedhere
\end{gather*}
\end{proof}
Using Lemma \ref{lem: balance}, we have
\[-c dQ_t 
=\int_\Rr\left(u_{xq} \sigma^S  + u_{xw} \sigma^P \right)\bar{m}(dx)dW_t +d\varpi_t;\]
that is,
\begin{gather*}
-cb^S dt - c\sigma^S  dW_t =  \left( \sigma^S \int_\Rr u_{xq}\bar{m}(dx)   +  \sigma^P \int_\Rr u_{xw}\bar{m}(dx) \right)dW_t  + d\varpi_t\\
=b^P dt +  \left(\sigma^S \int_\Rr u_{xq}\bar{m}(dx)  + \sigma^P \int_\Rr u_{xw}\bar{m}(dx)  + \sigma^P\right)dW_t.
\end{gather*}
Thus,
\begin{gather}\label{eq: price coefficients no potential}
b^P=-cb^S,  \\
\sigma^P =-\sigma^S \dfrac{c+ \int_\Rr u_{xq}\bar{m}(dx)}{1+ \int_\Rr u_{xw}\bar{m}(dx)}.\nonumber
\end{gather}
\subsection{Quadratic solutions to the Hamilton Jacobi equation}\label{subsection: quadratic solutions}
If $u$ is a second-degree polynomial with time-dependent coefficients, then 
\[\int_\Rr u_{xq}(\bX_t^*(x,\bar{q},\bar{w}),Q_t(\bar{q},\bar{w}),\varpi_t(\bar{q},\bar{w}),t)\bar{m}(dx)\]
and
\[\int_\Rr u_{xw}(\bX_t^*(x,\bar{q},\bar{w}),Q_t(\bar{q},\bar{w}),\varpi_t(\bar{q},\bar{w}),t)\bar{m}(dx)\]
 are deterministic functions of time. Accordingly, $b^P$ and $\sigma^P$ are given in feedback form by \eqref{eq: price coefficients no potential}, thus, consistent with the original assumption. Here, we investigate the linear-quadratic case that admits solutions of this form.\\
Now, we assume that the dynamics are affine; that is,
\begin{equation}\label{eq: affine dynamics supply price}
        \begin{cases}
                b^P(t,q,w)=b_0^P(t)+q b_1^P(t)+w b_2^P(t)\\
                b^S(t,q,w)=b_0^S(t)+q b_1^S(t)+w b_2^S(t)\\
                \sigma^P(t,q,w)=\sigma_0^P(t)+q \sigma _1^P(t)+w \sigma _2^P(t)\\
                \sigma^S(t,q,w)=\sigma_0^S(t) +q \sigma _1^S(t)+w \sigma _2^S(t).
        \end{cases}
\end{equation}
Then, \eqref{eq: price coefficients no potential} gives
\begin{align*}
        b^P_0&=-cb_0^S  , &\sigma^P_0 &=-\sigma_0^S \dfrac{c+  \int_\Rr u_{xq}\bar{m}(dx)}{1+  \int_\Rr u_{xw}\bar{m}(dx)}\\
        b^P_1&=-cb_1^S  , &\sigma^P_1 &=-\sigma_1^S \dfrac{c+ \int_\Rr u_{xq}\bar{m}(dx)}{1+  \int_\Rr u_{xw}\bar{m}(dx)}\\
        b^P_2&=-cb_2^S  , &\sigma^P_2 &=-\sigma_2^S \dfrac{c+ \int_\Rr u_{xq}\bar{m}(dx)}{1+ \int_\Rr u_{xw}\bar{m}(dx)}.
\end{align*}
Because all the terms in the Hamilton Jacobi equation 
are at most quadratic, we seek for solutions of the form 
\begin{align*}
u(t,x,q,w) = & a_0(t) + a_1^1(t)x + a_1^2(t)q + a_1^3(t)w\\
& + a_2^1(t)x^2 + a_2^2(t)xq + a_2^3(t)xw + a_2^4(t)q^2 + a_2^5(t)qw + a_2^6(t)w^2,
\end{align*}
where $a_i^j:[0,T]\to \Rr$. Therefore, the previous identities reduce to
\begin{align}\label{eq: price coefficients no potential u cuadratic}
        b^P_0&=-cb_0^S  , &\sigma^P_0 &=-\sigma_0^S \dfrac{c+ a_2^2 }{1+ a_2^3} \nonumber \\
        b^P_1&=-cb_1^S  , &\sigma^P_1 &=-\sigma_1^S \dfrac{c+ a_2^2 }{1+ a_2^3} \nonumber\\
        b^P_2&=-cb_2^S  , &\sigma^P_2 &=-\sigma_2^S \dfrac{c+ a_2^2 }{1+ a_2^3}.
\end{align}
Using \eqref{eq: price coefficients no potential u cuadratic} and grouping coefficients in the Hamilton Jacobi PDE, we obtain the following ODE system
\begin{align*}
 \dot{a}_2^1=&\tfrac{2 \left(a_2^1\right)^2}{c} 
 \\ 
  \dot{a}_2^2=&\tfrac{c^2 a_2^3 b_1^S-c a_2^2 b_1^S+2 a_2^1 a_2^2}{c} 
  \\ 
 \dot{a}_2^3=&\tfrac{c^2 a_2^3 b_2^S-c a_2^2 b_2^S+2 a_2^1+2 a_2^1 a_2^3}{c} 
 \\ 
 \dot{a}_1^1=&\tfrac{c^2 a_2^3 b_0^S-c a_2^2 b_0^S+2 a_1^1 a_2^1}{c} 
 \\
 \dot{a}_2^4=&c a_2^5 b_1^S-2 a_2^4 b_1^S+\tfrac{a_2^5 \left(a_2^2+c\right) \left(\sigma _1^S\right)^2}{a_2^3+1}+\tfrac{1}{4} \left(-\tfrac{4 a_2^6 \left(a_2^2+c\right)^2 \left(\sigma _1^S\right)^2}{\left(a_2^3+1\right)^2}-4 a_2^4 \left(\sigma _1^S\right)^2\right)+\tfrac{\left(a_2^2\right)^2}{2 c} 
 \\ 
 \dot{a}_2^5=&2 c a_2^6 b_1^S+c a_2^5 b_2^S-a_2^5 b_1^S-2 a_2^4 b_2^S+\tfrac{1}{2} \left(-\tfrac{4 a_2^6 \left(a_2^2+c\right)^2 \sigma _1^S \sigma _2^S}{\left(a_2^3+1\right)^2}-4 a_2^4 \sigma _1^S \sigma _2^S\right)
 \\&+\tfrac{2 a_2^5 \left(a_2^2+c\right) \sigma _1^S \sigma _2^S}{a_2^3+1}+\tfrac{a_2^2 \left(a_2^3+1\right)}{c} \\ 
 \dot{a}_2^6=&2 c a_2^6 b_2^S-a_2^5 b_2^S+\tfrac{1}{4} \left(-\tfrac{4 a_2^6 \left(a_2^2+c\right)^2 \left(\sigma _2^S\right)^2}{\left(a_2^3+1\right)^2}-4 a_2^4 \left(\sigma _2^S\right)^2\right)
 \\
 &+\tfrac{a_2^5 \left(a_2^2+c\right) \left(\sigma _2^S\right)^2}{a_2^3+1}+\tfrac{\left(a_2^3+1\right)^2}{2 c}
 \\  
 \dot{a}_0=&c a_1^3 b_0^S-a_1^2 b_0^S+\tfrac{a_2^5 \left(a_2^2+c\right) \left(\sigma _0^S\right)^2}{a_2^3+1}+\tfrac{1}{2} \left(-\tfrac{2 a_2^6 \left(a_2^2+c\right)^2 \left(\sigma _0^S\right)^2}{\left(a_2^3+1\right)^2}-2 a_2^4 \left(\sigma _0^S\right)^2\right)+\tfrac{\left(a_1^1\right)^2}{2 c} 
 \\ 
 \dot{a}_1^2=&c a_2^5 b_0^S+c a_1^3 b_1^S-2 a_2^4 b_0^S-a_1^2 b_1^S+\tfrac{2 a_2^5 \left(a_2^2+c\right) \sigma _0^S \sigma _1^S}{a_2^3+1}
 \\
 &+\tfrac{1}{2} \left(-\tfrac{4 a_2^6 \left(a_2^2+c\right)^2 \sigma _0^S \sigma _1^S}{\left(a_2^3+1\right)^2}-4 a_2^4 \sigma _0^S \sigma _1^S\right)+\tfrac{a_1^1 a_2^2}{c} 
 \\ 
 \dot{a}_1^3=&2 c a_2^6 b_0^S+c a_1^3 b_2^S-a_2^5 b_0^S-a_1^2 b_2^S+\tfrac{1}{2} \left(-\tfrac{4 a_2^6 \left(a_2^2+c\right)^2 \sigma _0^S \sigma _2^S}{\left(a_2^3+1\right)^2}-4 a_2^4 \sigma _0^S \sigma _2^S\right)
 \\
 &+\tfrac{2 a_2^5 \left(a_2^2+c\right) \sigma _0^S \sigma _2^S}{a_2^3+1}+\tfrac{a_1^1 \left(a_2^3+1\right)}{c} ,
\end{align*}
with terminal conditions
\begin{align*}
        a_0(T)&=\Psi(0,0,0)=c_0 &  a_1^1(T)&=D_{x}\Psi(0,0,0)=c_1^1\\
        a_1^2(T)&=D_{q}\Psi(0,0,0)=c_1^2 &       a_1^3(T)&=D_{w}\Psi(0,0,0)=c_1^3\\
        a_2^1(T)&=\tfrac{1}{2} D_{xx}\Psi(0,0,0)=c_2^1 &        a_2^2(T)&=D_{xq}\Psi(0,0,0)=c_2^2\\
        a_2^3(T)&=D_{xw}\Psi(0,0,0)=c_2^3 &      a_2^4(T)&=\tfrac{1}{2} D_{qq}\Psi(0,0,0)=c_2^4\\
        a_2^5(T)&=D_{qw}\Psi(0,0,0)=c_2^5 &       a_2^6(T)&=\tfrac{1}{2} D_{ww}\Psi(0,0,0)=c_2^6.
\end{align*}
While this system has a complex structure, it admits some simplifications. For example, the equation for $a_2^1$ is independent of other terms and has the solution 
\[a_2^1(t)= \tfrac{c c_2^1}{c+2c_2^1(T-t)}.\]
Moreover, we can determine $a_2^2$ and $a_2^3$ from the linear system
\[
        \tfrac{d}{dt} \begin{bmatrix} a_2^2\\ a_2^3 \end{bmatrix} = \begin{bmatrix} -b_1^S + \tfrac{2}{c} a_2^1 & c b_1^S\\ -b_2^S & c b_2^S +\tfrac{2}{c}a_2^1 \end{bmatrix}   \begin{bmatrix} a_2^2\\ a_2^3 \end{bmatrix} + \begin{bmatrix} 0\\ \tfrac{2}{c}a_2^1 \end{bmatrix}.
\]
Lemma \ref{lem: balance} takes the form
\begin{gather*}
d\Pi_t= \Big(a_2^2(t) \sigma^S(Q_t,\varpi_t,t) + a_2^3 (t)\sigma^P (Q_t,\varpi_t,t) \Big)dW_t.
\end{gather*}
Therefore,
\begin{gather*}
\Pi_t=\Pi_0 + \int_0^t  \Big(a_2^2(r) \sigma^S(Q_r,\varpi_r,r) + a_2^3 (r)\sigma^P (Q_r,\varpi_r,r) \Big)dW_r
\end{gather*}
where
\[\Pi_0=
a_1^1(0)+2a_2^1(0)\int_{\Rr}x\bar{m}(dx)+a_2^2(0)\bar{q}+a_2^3(0)\bar{w}.\]
Replacing the above in the balance condition at the initial time, that is $\bar{w}=-c\bar{q}-\Pi_0$, we obtain the initial condition for the price 
\begin{equation}\label{eq: price initial condition}
\bar{w}=\tfrac{-1}{1+a_2^3(0)}\Big(a_1^1(0) + 2a_2^1(0) \int_{\Rr}x\bar{m}(dx) + (a_2^2(0)+c) \bar{q}\Big).
\end{equation}
where $a_1^1$ can be obtained after solving for $a_2^1$, $a_2^2$ and $a_2^3$.\\
Now, we proceed with the price dynamics using the balance condition. Under linear dynamics, we have
\begin{align*}
Q_t&=-\tfrac{1}{c} \left(\varpi_t+\Pi_0\right)
\\
& -\tfrac{1}{c} \int_0^t  a_2^2(r) \Big(\sigma_0^S(r) +Q_r \sigma _1^S(r)+\varpi_r \sigma _2^S(r) \Big) + a_2^3 (r)\Big(\sigma_0^P(r)+Q_r \sigma _1^P(r)+\varpi_r \sigma _2^P(r) \Big) dW_r.
\end{align*}
Thus, replacing the price coefficients for \eqref{eq: price coefficients no potential u cuadratic}, we obtain
\begin{align*}
d\varpi_t=&-c\left( b_0^S(t)+ b_1^S(t)Q_t+ b_2^S(t)\varpi_t\right) dt \\
&- \tfrac{c+a_2^2(t)}{1+a_2^3(t)} \left(\sigma_0^S(t)+\sigma_1^S(t) Q_t+\sigma_2^S(t) \varpi_t\right)dW_t,\\
dQ_t=&b^S dt + \sigma^S dW_t,
\end{align*}
which determines the dynamics for the price.

\section{Simulation results}\label{section: simulation}
In this section, we consider the running cost corresponding to $c=1$; that is,
\[L(\bv)=\tfrac{1}{2} \bv^2\]
and terminal cost at time $T=1$ 
\[\Psi(x)=(x-\alpha)^2.\]
We take $\bar{m}$ to be a normal standard distribution; that is, with zero-mean and unit variance. We assume the dynamics for the normalized supply is mean-reverting
\[ dQ_t = (1-Q_t) dt + Q_t dW_t,\]
with initial condition $\bar{q}=1$. Therefore, the dynamics for the price becomes
\[ d\varpi_t = -(1-Q_t) dt - \tfrac{1+a_2^2}{1+a_2^3} Q_t dW_t,\]
with initial condition $\bar{w}$ given by \eqref{eq: price initial condition}, and $a_2^2$ and $a_2^3$ solve
\begin{align*}
  \dot{a}_2^2=& -a_2^3 + a_2^2 (1+2 a_2^1)
  \\ 
 \dot{a}_2^3=& 2 a_2^1(1+a_2^3),
\end{align*}
with terminal conditions $a_2^2(1)=0$ and $a_2^3(1)=0$. We observe that the coefficient multiplying $Q_t$ in the volatility of the price is now time-dependent.
\begin{figure}[h]
	\includegraphics[scale=1]{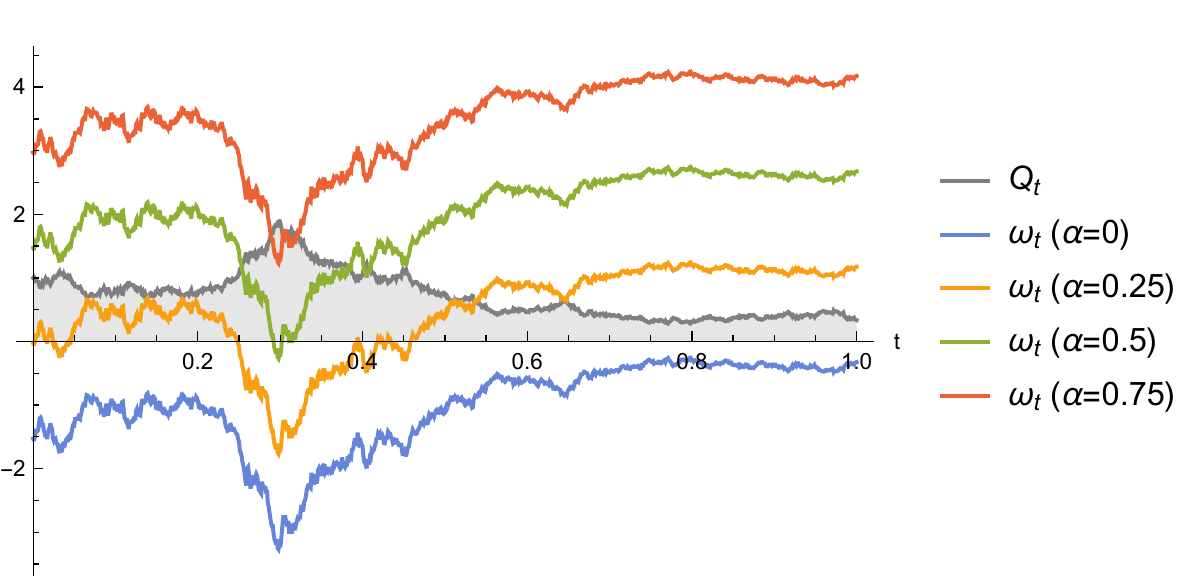}
	\caption{Supply vs. Price for the values $\alpha=0$, $\alpha=0.1$, $\alpha=0.25$, $\alpha=0.5$}
\label{fig:plots}
\end{figure}
For a fixed simulation of the supply, we compute the price for different values of $\alpha$. Agents begin with zero energy average. The results are displayed in Figure \ref{fig:plots}. As expected, the price is negatively correlated with the supply. Moreover, as the storage target increases, prices increase, which reflects the competition between agents who, on average, want to increase their storage.



\bibliographystyle{plain}
\IfFileExists{"/Users/gomesd/mfgDGOFFICE.bib"}{
	\bibliography{/Users/gomesd/mfgDGOFFICE.bib}}{\bibliography{mfg.bib}}

\end{document}